\documentclass[reqno]{amsart}
\usepackage{hyperref}

\begin{document}
\title[Uniformly monotone operators]
{Examples for uniformly monotone operators arising in nonlinear elliptic and parabolic problems}

\author[\'A. Besenyei]
{\'Ad\'am Besenyei} 

\address{\'Ad\'am Besenyei \newline
Department of Applied Analysis\\
E\"otv\"os Lor\'and University, 
H-1117 Budapest P\'azm\'any P. s. 1/C, Hungary}
\email{badam@cs.elte.hu}

\thanks{Supported by grant OTKA T 049819 from the
Hungarian National Foundation \hfill\break\indent for Scientific Research}
\subjclass[2000]{47H05, 35J60, 35K55}
\keywords{uniformly monotone operators; nonlinear parabolic and \hfill\break\indent elliptic problems}

\date{\today}

\begin{abstract}
We show some examples for uniformly monotone operators arising in weak formulation of nonlinear elliptic and parabolic problems. Besides the classical $p$-Laplacian some other less known examples are given which might be of interest because of applications.
\end{abstract}

\maketitle

\numberwithin{equation}{section}
\newtheorem{theorem}{Theorem}[section]
\newtheorem{proposition}[theorem]{Proposition}
\newtheorem{remark}[theorem]{Remark}
\newtheorem{lemma}[theorem]{Lemma}

\newcommand{\const}{\mathrm{const}}

\hfuzz2pt

\setlength{\arraycolsep}{.13889em}

\frenchspacing

\section{Introduction}

The aim of the present paper is to show several examples for uniformly monotone oprators arising in weak formulation of nonlinear elliptic and parabolic problems. Let $X$ be a normed space and denote by $X^*$ its dual, further, by $\langle\cdot,\cdot\rangle$ the pairing between $X^*$ and $X$. Then, an operator $A\colon X\to X^*$ is called uniformly monotone (following the terminology of \cite{zeidler}) if there exist $p\geq2$, $\gamma>0$ such that 
\begin{equation}\label{unif}
\langle A(u_1)-A(u_2), u_1-u_2\rangle\geq\gamma\cdot \|u_1-u_2\|_X^p.
\end{equation}
for all $u_1,u_2\in X$.
In what follows, we study operators which are obtained by considering the weak formulation of an elliptic or parabolic equation or system with some boundary conditions, see, e.g., \cite{bes1, bes2, simon}. Namely, in the elliptic case let $X$ be a linear subspace of $W^{1,p}(\Omega)$, where $\Omega\subset\mathbb{R}^n$ is bounded (with sufficiently smooth boundary), $p\geq2$, and consider operator $A\colon X\to X^*$ defined by
\begin{equation}\label{oper}
\langle A(u),v\rangle=\int_\Omega \Bigl(\sum_{i=1}^na_i(x,u(x),Du(x))D_iv(x)+a_0(x,u(x),Du(x))v(x)\Bigr)\,dx,
\end{equation}
where $D_i$ denotes the distributional derivative with respect to the $i$-th variable and $D=(D_1,\dotsc,D_n)$. 

The parabolic case is a minor modification of this, $X$ may be chosen as $L^p(0,T; V)$ (see, e.g., \cite{zeidler}), where $V$ is a linear subspace of $W^{1,p}(\Omega)$ and $0<T<\infty$, further, functions $a_i$ may depend on variable $t$, and in \eqref{oper} one integrates on $(0,T)\times \Omega$.

Supposing the uniform monotonicity (and some other conditions) for an operator like this, one can prove uniqueness of weak solutions, continuous dependence of the solutions on datas and for parabolic equations one can obtain results on stabilization of weak solutions, see, e.g., \cite{bes3, bes4, simon}. 

The well known example for such operator is $a_i(x,\xi)=\xi_i|(\xi_1,\dotsc,\xi_n)|^{p-2}$ ($i=1,\dotsc,n)$, $a_0(x,\xi)=\xi_0|\xi_0|^{p-2}$ where $\xi=(\xi_0,\xi_1,\dotsc,\xi_n)$ refers to $(u,D_1u,\dotsc, D_nu)$. In this case, assuming homogeneous boundary condition, Gauss's theorem yields that operator \eqref{oper} is the weak form of operator $u\mapsto\Delta_p u+u|u|^{p-2}$ where 
\[\Delta_pu=\mathrm{div}\,(\mathrm{grad}\,u|\mathrm{grad}\,u|^{p-2})\] 
is the $p$-Laplacian. Note that in case $X=W_0^{1,p}(\Omega)$ (i.e. the closure of $C_0^\infty(\Omega)$ in $W^{1,p}(\Omega)$) $\Delta_p$ is also uniformly monotone since, due to Poincar\'e's inequality, an equivalent norm can be introduced in $W_0^{1,p}(\Omega)$, see \cite{adams}.

In \cite{bes3, bes4} we considered a nonlinear system modelling fluid flow in porous media. In that case functions $a_i$ did not depend on $(\xi_0,\dotsc,\xi_k)$ if $i>k$. The present paper was motivated by such system. We shall show a variety of examples for uniformly monotone operators, including functions $a_i$ of that type. In the next section we shall formulate and prove a result of \cite{dub} which is a sufficient condition for uniform monotonicity and this will be applied to examples in Section 3. For further details on operators of monotone type, see \cite{lions, zeidler}.


\section{Sufficient condition}

Let $\Omega\subset\mathbb{R}^n$ be bounded and $X$ be a linear subspace of $W^{1,p}(\Omega)$ ($p\geq2$) and use the notations introduced in the prevoius section.
We define operator $A\colon X\to X^*$ by the formula \eqref{oper}.
A weak form of an elliptic problem may be written as $A(u)=F$
where $F\in X^*$ (for example $\langle F,v\rangle =\int_\Omega f(x)v(x)\,dx$ with some $f\in L^q(\Omega)$ where $\frac1p+\frac1q=1$). (In the parabolic case one has the form $D_tu+Au=F$, where $D_t$ denotes the distributional derivative with respect to the variable $t$, further, $X$ and $A$ are modified as explained in the introduction.)
This kind of abstract problems have an extended classical theory (see, e.g., \cite{lions, zeidler}). Existence and uniqueness of solutions can be guaranteed by supposing the following conditions:

\begin{enumerate}
\item[(A1)] Functions $a_i\colon \Omega\times\mathbb{R}^{n+1}\to\mathbb{R}$ ($i=0,\dotsc,n$) are of Carath\'eodory property type, i.e., $a_i(x,\xi)$ is measurable in $x\in\Omega$ for all $\xi\in\mathbb{R}^{n+1}$, and continuous in $\xi\in\mathbb{R}^{n+1}$ for a.a. $x\in\Omega$.

\item[(A2)] There exist a constant $c>0$ and a function $k\in L^q(\Omega)$ such that for a.a. $x\in\Omega$ and all $\xi\in\mathbb{R}^{n+1}$ 
\[|a_i(x,\xi)|\leq c\cdot |\xi|^{p-1}+k(x),\qquad i=0,\dotsc, n.\]

\item[(A3)] There exists a constant $C>0$ such that for a.a. $x\in\Omega$ and all $\xi,\Tilde\xi\in\mathbb{R}^n$
\[
\sum_{i=0}^n(a_i(x,\xi)-a_i(x,\Tilde\xi))(\xi_i-\Tilde\xi_i)\geq C\cdot |\xi-\Tilde \xi|^p.
\]

\end{enumerate}

Clearly, the third condition implies \eqref{unif} for all $u_1,u_2\in X$ with $\gamma=C$ thus (A3) ensures the uniform monotonicity of operator $A$.

Now we recall a result of \cite{dub} which is a sufficient condition on functions $a_i$ which guarantees condition (A3).

\begin{proposition}\label{suff} Suppose that $p\geq2$ and $a_i$ is continuously differentiable in variable $\xi$ for all $i=0,\dotsc,n$. Further, assume that there exists a constant $\delta>0$ such that for a.a. $x\in\Omega$, all $\xi\in\mathbb{R}^{n+1}$ and all $(z_0,\dotsc,z_n)\in\mathbb{R}^{n+1}$,
\begin{equation}\label{suff2}
\sum_{j=0}^n\sum_{i=0}^nD_ja_i(x,\xi)z_iz_j\geq\delta\cdot\sum_{i=0}^n |\xi_i|^{p-2}z_i^2.
\end{equation}
Then condition (A3) holds.
\end{proposition}

To prove this assertion we shall apply the following inequality \eqref{intineq} from \cite{dub}. For the convenience we present the proof of it below.

\begin{lemma} Let $a, b$ be arbitrary and $s\geq 0$ real numbers. Then 
\begin{equation}\label{intineq}
\int_0^1|a+\tau b|^s\, d\tau\geq
\frac{|b|^s}{2^s(s+1)}.
\end{equation}
\end{lemma}

\begin{proof} The case $b=0$ is obvious. Now let $b\neq0$, then by homogenity suffices to verify the inequality for $b=1$.
Further, without loss of generality we may assume $a<0$. We have to cases. If $a+1>0$, then by dividing interval $[0,1]$ into two parts with respect to the sign of $c+\tau$, and by integrating separately we obtain
\begin{align*}
\int_0^1|a+\tau|^s \,d\tau&=
\int_0^{-a}(-a-\tau)^s\, d\tau+
\int_{-a}^1(a+\tau)^s \, d\tau\\&=
\frac1{s+1}\left((-a)^{s+1}+(a+1)^{s+1}\right)\\&\geq
\frac{1}{2^s(s+1)}.
\end{align*}
In the last estimate we used inequality $|\alpha+\beta|^{s+1}\leq2^{-s}\left(|\alpha|^{s+1}+|\beta|^{s+1}\right)$. In the other case $a+\tau$ is negative for all $\tau\in [0,1]$ thus
\[\int_0^1|a+\tau|^s \,d\tau\geq\int_0^1|\tau|^s
\,d\tau=
\frac{1}{s+1}
\geq\frac1{2^s(s+1)}.\] 
Note that \eqref{intineq} is sharp, we have equality if $\dfrac{a}{b}=-\dfrac12$.
\end{proof}

Now we prove Proposition \ref{suff}. We follow the proof of \cite{dub}.

\begin{proof}
Fix $x\in\Omega$, $\xi,\Tilde\xi \in \mathbb{R}^{n+1}$ and
define functions $f_i\colon[0,1]\to\mathbb{R}$ by $f_i(\tau)=a_i(x,\Tilde\xi+\tau(\xi-\Tilde\xi))$ ($i=0,\dotsc,n)$. Then by applying assumption \eqref{suff2} and inequality \eqref{intineq} we may deduce
\begin{align*}
\sum_{i=0}^n(a_i(x,\xi)-a_i(x,\Tilde\xi))(\xi_i-\Tilde\xi_i)&=\sum_{i=0}^n(f_i(1)-f_i(0))(\xi_i-\Tilde\xi_i)\\&=\sum_{i=0}^n\int_0^1\sum_{j=0}^nD_ja_i(\Tilde\xi+\tau(\xi-\Tilde\xi))(\xi_j-\Tilde\xi_j)(\xi_i-\Tilde\xi_i)\,d\tau\\&\geq\delta\cdot\sum_{i=0}^n\int_0^1|\Tilde\xi+\tau(\xi-\Tilde\xi)|^{p-2}(\xi_i-\Tilde\xi_i)^2d\tau\\&\geq
\dfrac{\delta}{2^{p-2}(p-1)}|\xi-\Tilde\xi|^p.
\end{align*}
Whence after integration we conclude 
\[\langle A(u_1)-A(u_2),u_1-u_2\rangle\geq\displaystyle \dfrac{\delta}{2^{p-2}(p-1)}\|u_1-u_2\|_X^p.\]Thus condition A3 holds with $\displaystyle C=\dfrac{\delta}{2^{p-2}(p-1)}$.
\end{proof}


\section{Examples}

Now we show some examples of uniformly monotone operators which fulfil also conditions (A1), (A2). For simplicity, we consider examples not depending on variable $x$. In the sequel we always suppose $p\geq 2$.

\par \textit{Example 1}\, Let $a_i(\xi)=\xi_i|\xi_i|^{p-2}$ ($i=0,\dotsc,n$). Then 
\[\langle A(u),v\rangle =\int_\Omega\left(\sum_{i=1}^n D_iuD_iv|D_iu|^{p-2}+uv|u|^{p-2}\right)\,dx.\] 
Note that functions $a_i$ obviously fulfil conditions (A1), (A2).  
Now simple calculations yield $D_ia_i(\xi)=(p-1)|\xi_i|^{p-2}$ and $D_ja_i(\xi)=0$ ($j\neq i$). Hence
\[\sum_{j=0}^n\sum_{i=0}^nD_ja_i(\xi)z_iz_j=(p-1)\sum_{i=0}^n|\xi_i|^{p-2}z_i^2\]
thus, by Proposition \ref{suff}, condition (A3) holds as well.


\par \textit{Example 2}\, Now let $a_i(\xi)=\xi_i|(\xi_1,\dotsc,\xi_n)|^{p-2}$ for $i=1,\dotsc, n$ and $a_0(\xi)=\xi_0|\xi_0|^{p-2}$. In this case \[\langle A(u),v \rangle=\int_\Omega\Bigl(\sum_{i=1}^n D_iuD_iv|Du|^{p-2}+uv|u|^{p-2}\Bigr),\] 
i.e., $A$ is the weak form of operator $u\mapsto\Delta_p u+u|u|^{p-2}$ mentioned in the introduction. Obviously, functions $a_i$ satisfy conditions (A1), (A2). 
It is easy to see that 
\begin{center}
\[
\left\{\begin{array}{rcll}
D_ja_i(\xi)&=&(p-2)\xi_j\xi_i|(\xi_1,\dotsc,\xi_n)|^{p-4},& \mbox{ for } i,j>0,\, i\neq j;\\ 
D_ia_i(\xi)&=&|(\xi_1,\dotsc,\xi_n)|^{p-2}+(p-2)\xi_i^2|(\xi_1,\dotsc,\xi_n)|^{p-4},& \mbox{ for } i>0;\\
D_ja_0(\xi)&=&D_0a_i(\xi)=0,& \mbox{ for } j>0, i>0;\\
D_0a_0(\xi)&=&(p-1)|\xi_0|^{p-2}.&
\end{array}\right.\]
\end{center}
Hence
\begin{align*}
\sum_{j=0}^n\sum_{i=0}^nD_ja_iz_iz_j&=\sum_{i=1}^n|(\xi_1,\dotsc,\xi_n)|^{p-2}z_i^2+(p-1)|\xi_0|^{p-2}z_0^2\\&\quad+(p-2)|(\xi_1,\dotsc,\xi_n)|^{p-4}\cdot\sum_{j=1}^n\sum_{i=1}^n\xi_i\xi_jz_iz_j\\&=\sum_{i=1}^n|(\xi_1,\dotsc,\xi_n)|^{p-2}z_i^2+(p-1)|\xi_0|^{p-2}z_0^2\\&\quad+(p-2)|(\xi_1,\dotsc,\xi_n)|^{p-4}\cdot\left(\sum_{i=1}^n\xi_iz_i\right)^2\\&\geq\sum_{i=0}^n|\xi_i|^{p-2}z_i^2
\end{align*}
thus from Proposition \ref{suff} it follows that $A$ is uniformly monotone.


\par \textit{Example 3}\, Let $a_i(\xi)=\xi_i|\xi|^{p-2}+g_i(\xi)$ ($i=0,\dotsc, n)$, where functions  $g_i\colon\mathbb{R}^{n+1}\to\mathbb{R}$ are continuous, further, there exist positive constants $c,\,\varepsilon$ such that 
\begin{equation}\label{gprop}
|g_i(\xi)|\leq c\cdot |\xi|^{p-1}\quad \mbox{ and }\quad
|D_jg_i(\xi)|< \dfrac1{n+1+\varepsilon}\cdot |\xi|^{p-2}
\end{equation}
for all $\xi=(\xi_0,\dotsc,\xi_n)\in\mathbb{R}^{n+1}$ ($i,j=0,\dotsc,n$). It is clear that conditions (A1), (A2) hold. By using Example 2 and the inequality $|\alpha\beta|\leq\dfrac12(\alpha^2+\beta^2)$ one obtains 
\begin{align*}
\sum_{j=0}^n\sum_{i=0}^nD_ja_i(\xi)z_iz_j&\geq\sum_{i=0}^n|\xi|^{p-2}z_i^2-\dfrac12\sum_{j=0}^n\sum_{i=0}^n|D_jg_i(\xi)|(z_i^2+z_j^2)\geq\\&\geq\sum_{i=0}^n|\xi|^{p-2}z_i^2-(n+1)\sum_{i=0}^n\frac1{n+1+\varepsilon}|\xi|^{p-2}z_i^2\\&=\sum_{i=0}^n\frac{\varepsilon}{n+1+\varepsilon} |\xi|^{p-2}z_i^2
\end{align*}
which implies condition (A3). As an example for functions $g_i$ with the properties \eqref{gprop}, consider, e.g.,   \[g_i(\xi)=\dfrac{1}{(n+1+\varepsilon)\cdot\max\{\alpha_j,j=0,\dotsc,n\}}\prod_{j=0}^n|\xi_j|^{\alpha_j}\] 
where $\alpha_j\geq0$ for all $j=0,\dotsc, n$ and $\displaystyle\sum_{j=0}^n\alpha_j=p-1$.


\par \textit{Example 4} Now we show example for the system considered in \cite{bes3, bes4} which was mentioned in the introduction. Suppose $2\leq p\leq 4$ and let $a_i(\xi)=\xi_i|\xi|^{p-2}$ for $0\leq i\leq k\leq n$ and $a_i(\xi)=\xi_i|(\xi_{k+1},\dotsc,\xi_n)|^{p-2}$ for $k<i\leq n $.
We show that these functions fulfil condition (A3) ((A1) and (A2) obviously hold). Now for brevity let $\zeta=(\xi_{k+1},\dotsc,\xi_n)$. 
Clearly,
\begin{center}
\[\left\{
\begin{array}{rcll}
D_ja_i(\xi)&=&(p-2)\xi_i\xi_j|\xi|^{p-4},& \mbox{ for } 0\leq i\leq k, 0\leq j\leq n, j\neq i;\\
D_ja_i(\xi)&=&(p-2)\xi_i\xi_j|\zeta|^{p-4},& \mbox{ for } k< i\leq n, k<j<n, j\neq i;\\
D_ja_i(\xi)&=&0,& \mbox{ for } k< i\leq n, 0\leq j\leq k\\
D_ia_i(\xi)&=&|\xi|^{p-2}+(p-2)\xi_i^2|\xi|^{p-4},& \mbox{ for } 0\leq i\leq k;\\
D_ia_i(\xi)&=&|\zeta|^{p-2}+(p-2)\xi_i^2|\zeta|^{p-4},& \mbox{ for } k<i\leq n.
\end{array}\right.
\]
\end{center}
Then
\begin{align*}
&\sum_{j=0}^n\sum_{i=0}^nD_ja_i(\xi)z_iz_j\\&=\sum_{i=0}^k|\xi|^{p-2}z_i^2+(p-2)|\xi|^{p-4}\sum_{j=0}^n\sum_{i=0}^k\xi_i\xi_jz_iz_j\\&\quad+\sum_{i=k+1}^n|\zeta|^{p-2}z_i^2+(p-2)|\zeta|^{p-4}\sum_{j=k+1}^n\sum_{i=k+1}^n\xi_i\xi_jz_iz_j\\&=
\sum_{i=0}^k|\xi|^{p-2}z_i^2+\sum_{i=k+1}^n|\zeta|^{p-2}z_i^2+(p-2)|\xi|^{p-4}\left(\sum_{i=0}^k\xi_iz_i\right)^2\\&\quad+(p-2)|\zeta|^{p-4}\left(\sum_{i=k+1}^n\xi_iz_i\right)^2+(p-2)|\xi|^{p-4}\sum_{j=k+1}^n\sum_{i=0}^k\xi_i\xi_jz_iz_j.
\end{align*}
By using the estimate
\begin{align*}
\sum_{j=k+1}^n\sum_{i=0}^k\xi_i\xi_jz_iz_j&=\left(\sum_{j=k+1}^n\xi_jz_j\right)\left(\sum_{i=0}^k\xi_iz_i\right)\\&\geq-\dfrac12\left(\sum_{i=k+1}^n\xi_iz_i\right)^2-\dfrac12\left(\sum_{i=0}^k\xi_iz_i\right)^2.
\end{align*}
and the fact that $|\zeta|^{p-4}\geq|\xi|^{p-4}$ (since $p\leq4$) we conclude 
\begin{align*}
&\sum_{j=0}^n\sum_{i=0}^nD_ja_i(\xi)z_iz_j\\&=
\sum_{i=0}^k|\xi|^{p-2}z_i^2+\sum_{i=k+1}^n|\zeta|^{p-2}z_i^2+\dfrac12(p-2)|\xi|^{p-4}\left(\sum_{i=k+1}^n\xi_iz_i\right)^2\\&\quad+\dfrac12(p-2)|\xi|^{p-4}\left(\sum_{i=0}^k\xi_iz_i\right)^2\geq\sum_{i=0}^n|\xi_i|^{p-2}z_i^2.
\end{align*}
Now Proposition \ref{suff} yields the uniform monotonicity of $A$.

\par In case $p>4$ one may consider, e.g., the following functions: 
\setlength{\arraycolsep}{.13889em}
\begin{eqnarray*}
a_i(\xi)&=&\xi_i|(\xi_0,\dotsc,\xi_k)|^{p-2}+\xi_i|\xi|^{r-2} \;\mbox{ if }\; 0\leq i\leq k\leq n,\\ a_i(\xi)&=&\xi_i|(\xi_{k+1},\dotsc,\xi_n)|^{p-2}+\xi_i|(\xi_{k+1},\dotsc,\xi_n)|^{r-2} \;\mbox{ if }\; k<i<n,
\end{eqnarray*}
where $2\leq r\leq4$. 
By using the previous examples it is not difficult to show that these functions satisfy condition (A3).


\par \textit{Example 5}\, Now let $a_i(\xi)=\xi_i|\xi_i|^{p-2}+\displaystyle\prod_{k=0}^n\xi_k|\xi_k|^{p-2}\cdot h_i(\xi)$ ($i=0,\dotsc,n$) where functions $h_i\colon\mathbb{R}^{n+1}\to\mathbb{R}$ are differentiable and have compact support $S_i$ ($i=0,\dotsc,n$), denote $\displaystyle S=\bigcup_{i=0}^nS_i$. Further, let \[\alpha=\displaystyle p\max\left\{\sup_{\xi\in S}|\xi|^{(n+1)(p-1)},1\right\}\cdot\max\left\{\sup_S|h|,\sup_S|Dh|\right\}.\]
We show that if $\alpha$ is sufficiently small then functions $a_i$ fulfil condition (A3) ((A1) is obvious and due to the compact supports (A2) is also satisfied).
Observe that 
\[D_ia_i(\xi)=(p-1)|\xi_i|^{p-2}+(p-1)|\xi_i|^{p-2}\prod_{k\neq i}\xi_k|\xi_k|^{p-2}\cdot h_i(\xi)+\prod_{k=0}^n\xi_k|\xi_k|^{p-2}\cdot D_i h(\xi)\]
thus $D_ia_i(\xi)\geq (p-1)|\xi_i|^{p-2}-\alpha|\xi_i|^{p-2}$. In addition for $j\neq  i$
\[D_ja_i(\xi)=(p-1)|\xi_j|^{p-2}\prod_{k\neq j}\xi_k|\xi_k|^{p-2}\cdot h_i+\prod_{k=0}^n\xi_k|\xi_k|^{p-2}\cdot D_jh_i(\xi)\]
thus $|D_ja_i(\xi)|\leq \alpha\cdot |\xi_i|^{p-2}$ and $|D_ja_i(\xi)|\leq \alpha\cdot |\xi_j|^{p-2}$ hence \[|D_ja_i(\xi)z_iz_j|\leq\alpha\cdot\left(|\xi_i|^{p-2}z_i^2+|\xi_j|^{p-2}z_j^2\right).\] 
Therefore,
\begin{align*}
\sum_{j=0}^n\sum_{i=0}^nD_ja_i(\xi)z_iz_j&\geq (p-1-\alpha)\sum_{i=0}^n|\xi_i|^{p-2}z_i^2-\alpha\sum_{j=0}^n\sum_{i=0}^n\left(|\xi_i|^{p-2}z_i^2+|\xi_j|^{p-2}z_j^2\right)\\&\geq(p-1-\alpha)\sum_{i=0}^n|\xi_i|^{p-2}z_i^2-2n\alpha\sum_{i=0}^n|\xi_i|^{p-2}z_i^2\\&=(p-1-(2n+1)\alpha)\sum_{i=0}^n|\xi_i|^{p-2}z_i^2.
\end{align*}
Hence \eqref{suff2} holds provided $\alpha$ is sufficiently small which implies condition (A3).


\small{

\end{document}
\begin{thebibliography}{10}

\bibitem{adams}R.~A.~Adams, \textit{Sobolev spaces}, Academic
Press, New York~-~San Francisco~-~London, 1975.

\bibitem{bes1}\'A. Besenyei, On systems of parabolic functional
differential equations, \textit{Annales Univ.~
Sci.~Budapest}, \textbf{47} (2004), 143-160.

\bibitem{bes2}\'A. Besenyei, \textit{Existence of solutions of a nonlinear system modelling fluid flow in porous media}, Electron. J. Diff. Eqns., Vol. 2006(2006), No. 153, pp. 1-19., 

\bibitem{bes3}\'A. Besenyei, Stabilization of solutions to a nonlinear system modelling fluid flow in porous media, \textit{Annales Univ. Sci. Budapest. E\"otv\"os Sect. Math.}, \textbf{49} (2006), 115--136.

\bibitem{bes4}\'A. Besenyei, On a nonlinear system containing nonlocal terms related to a fluid flow model, \textit{E. J. Qualitative Theory of Diff. Equ., Proc. 8'th Coll. Qualitative Theory of Diff. Equ.}, No. 3. (2008), 1--13. 

\bibitem{dub} Yu. A. Dubinskiy, Nonlinear elliptic and parabolic equations (in Russian), in: Modern problems in mathematics, Vol. 9., Moscow, 1976.

\bibitem{lions}J.~L.~Lions, \textit{Quelques méthodes de résolution des
probl\`emes aux limites non linéaires}, Dunod, Gauthier-Villars,
Paris, 1969.

\bibitem{simon}L. Simon, Application of monotone type operators to parabolic and functional parabolic PDE's, In: C. M. Dafermos, M. Pokorn\'y (eds), \textit{Handbook of Differential Equations: Evolutionary Equations}, vol 4., North-Holland, Amsterdam, 2008., 267--321.


\bibitem{zeidler}E.~Zeidler, \textit{Nonlinear functional analysis
and its applications II}, Springer, 1990.

\end{thebibliography}
